\documentclass{proc-l}

\usepackage{amssymb}

\newtheorem{theorem}{Theorem}[section]

\newtheorem{proposition}[theorem]{Proposition}
\newtheorem{corollary}[theorem]{Corollary}

\theoremstyle{definition}

\theoremstyle{remark}
\newtheorem{remark}[theorem]{Remark}

\numberwithin{equation}{section}

\begin{document}

\title{Reflection formulas for order derivatives of Bessel functions}


\author{J. L. Gonz\'{a}lez-Santander}
\address{C/ Ovidi Montllor i Mengual 7, pta. 9. 46017, Valencia, Spain.}
\email{juanluis.gonzalezsantander@gmail.com}



\date{}


\commby{Mourad Ismail}

\begin{abstract}
From new integral representations of the $n$-th derivative of Bessel
functions with respect to the order, we derive some reflection formulas for
the first and second order derivative of $J_{\nu }\left( t\right) $ and $%
Y_{\nu }\left( t\right) $ for integral order, and for the $n$-th order
derivative of $I_{\nu }\left( t\right) $ and $K_{\nu }\left( t\right) $ for
arbitrary real order. As an application of the reflection formulas obtained
for the first order derivative, we extend some formulas given in the
literature to negative integral order. Also, as a by-product, we calculate
an integral which does not seem to be reported in the literature. 
\end{abstract}

\maketitle


\section{Introduction}

Bessel functions are the canonical solutions $y\left( t\right) $ of Bessel's
differential equation:%
\begin{equation}
t^{2}y^{\prime \prime }+t\,y^{\prime }+\left( t^{2}-\nu ^{2}\right) y=0,
\label{Eqn_Bessel}
\end{equation}%
where $\nu $ denotes the order of the Bessel function. This equation arises
when finding separable solutions of Laplace equation in cylindrical
coordinates, as well as in Helmholtz equation in spherical coordinates \cite[%
Chap. 6]{Lebedev}. The general solution of (\ref{Eqn_Bessel})\ is a linear
combination of the \textit{Bessel functions of the first and second kind},
i.e.\textit{\ }$J_{\nu }\left( t\right) $ and $Y_{\nu }\left( t\right) $
respectively. In the case of pure imaginary argument, the solutions to the Bessel
equations are called \textit{modified Bessel functions of the first and
second kind, }$I_{\nu }\left( t\right) $ and $K_{\nu }\left( t\right) $
respectively. 
Despite the fact the properties of the Bessel functions have been studied
extensively in the literature \cite{Watson, Andrews}, studies about
successive derivatives of the Bessel functions with respect to the order $%
\nu $ are relatively scarce. For nonnegative integral order $\nu =m$, we
find in the literature the following expressions in terms of finite sums of
Bessel functions \cite[Eqn. 10.15.3\&4]{NIST}%
\begin{equation}
\left. \frac{\partial J_{\nu }\left( t\right) }{\partial \nu }\right\vert
_{\nu =m}=\frac{\pi }{2}Y_{m}\left( t\right) +\frac{m!}{2}\sum_{k=0}^{m-1}%
\frac{J_{k}\left( t\right) }{k!\left( m-k\right) }\left( \frac{t}{2}\right)
^{k-m},  \label{DJ_integral}
\end{equation}%
and%
\begin{equation}
\left. \frac{\partial Y_{\nu }\left( t\right) }{\partial \nu }\right\vert
_{\nu =m}=-\frac{\pi }{2}J_{m}\left( t\right) +\frac{m!}{2}\sum_{k=0}^{m-1}%
\frac{Y_{k}\left( t\right) }{k!\left( m-k\right) }\left( \frac{t}{2}\right)
^{k-m}.  \label{DY_integral}
\end{equation}

For modified Bessel functions, we have \cite[Eqn. 10.38.3\&4]{NIST}%
\begin{equation}
\left. \frac{\partial I_{\nu }\left( t\right) }{\partial \nu }\right\vert
_{\nu =m}=\left( -1\right) ^{m}\left[ -K_{m}\left( t\right) +\frac{m!}{2}%
\sum_{k=0}^{m-1}\frac{\left( -1\right) ^{k}I_{k}\left( t\right) }{k!\left(
m-k\right) }\left( \frac{t}{2}\right) ^{k-m}\right] ,  \label{DI_integral}
\end{equation}%
and%
\begin{equation}
\left. \frac{\partial K_{\nu }\left( t\right) }{\partial \nu }\right\vert
_{\nu =m}=\frac{m!}{2}\sum_{k=0}^{m-1}\frac{K_{k}\left( t\right) }{k!\left(
m-k\right) }\left( \frac{t}{2}\right) ^{k-m}.  \label{DK_integral}
\end{equation}

Also, for the $n$-th derivative of the Bessel function of the first kind
with respect to the order, we find in \cite{Sesma}\ a more complex
expression in series form.

Regarding integral representations of the derivative of $J_{\nu }\left(
t\right) $ and $I_{\nu }\left( t\right) $ with respect to the order, we find
in \cite{Apelblat} $\forall \Re \nu >0$,%
\begin{equation}
\frac{\partial J_{\nu }\left( t\right) }{\partial \nu }=\pi \nu
\int_{0}^{\pi /2}\tan \theta \ Y_{0}\left( t\sin ^{2}\theta \right) J_{\nu
}\left( t\cos ^{2}\theta \right) d\theta ,  \label{DJnu_int_Apelblat}
\end{equation}%
and%
\begin{equation}
\frac{\partial I_{\nu }\left( t\right) }{\partial \nu }=-2\nu \int_{0}^{\pi
/2}\tan \theta \ K_{0}\left( t\sin ^{2}\theta \right) I_{\nu }\left( t\cos
^{2}\theta \right) d\theta .  \label{DInu_int_Apelblat}
\end{equation}

Other integral representations of the order derivative of $J_{\nu }\left(
z\right) $ and $Y_{\nu }\left( z\right) $ are given in \cite{Dunster} for $%
\nu >0$ and $t\neq 0$,$\left\vert \mathrm{arg}\ t\right\vert \leq \pi $,
which read as,%
\begin{equation}
\frac{\partial J_{\nu }\left( t\right) }{\partial \nu }=\pi \nu \left[
Y_{\nu }\left( t\right) \int_{0}^{t}\frac{J_{\nu }^{2}\left( z\right) }{t}%
dz+J_{\nu }\left( t\right) \int_{t}^{\infty }\frac{J_{\nu }\left( z\right)
Y_{\nu }\left( z\right) }{z}dz\right] ,  \label{DJnu_int_Dunster}
\end{equation}%
and%
\begin{eqnarray}
&&\frac{\partial Y_{\nu }\left( t\right) }{\partial \nu }
\label{DYnu_int_Dunster} \\
&=&\pi \nu \left[ J_{\nu }\left( t\right) \left( \int_{t}^{\infty }\frac{%
Y_{\nu }^{2}\left( z\right) }{z}dz-\frac{1}{2\nu }\right) -Y_{\nu }\left(
t\right) \int_{t}^{\infty }\frac{J_{\nu }\left( z\right) Y_{\nu }\left(
z\right) }{z}dz\right] .  \notag
\end{eqnarray}

Recently, in \cite{DBesselJL}, we find the following integral
representations of the derivatives of the modified Bessel functions $I_{\nu
}\left( t\right) $ and $K_{\nu }\left( t\right) $ with respect to the order
for $\nu >0$ and $t\neq 0$,$\left\vert \mathrm{arg}\ t\right\vert \leq \pi $,%
\begin{equation}
\frac{\partial I_{\nu }\left( t\right) }{\partial \nu }=-2\nu \left[ I_{\nu
}\left( t\right) \int_{t}^{\infty }\frac{K_{\nu }\left( z\right) I_{\nu
}\left( z\right) }{z}dz+K_{\nu }\left( t\right) \int_{0}^{t}\frac{I_{\nu
}^{2}\left( z\right) }{z}dz\right] ,  \label{DInu_int}
\end{equation}%
and%
\begin{equation}
\frac{\partial K_{\nu }\left( t\right) }{\partial \nu }=2\nu \left[ K_{\nu
}\left( t\right) \int_{t}^{\infty }\frac{I_{\nu }\left( z\right) K_{\nu
}\left( z\right) }{z}dz-I_{\nu }\left( t\right) \int_{t}^{\infty }\frac{%
K_{\nu }^{2}\left( z\right) }{z}dz\right] .  \label{DKnu_int}
\end{equation}

The great advantage of the integral expressions (\ref{DJnu_int_Dunster})-(%
\ref{DKnu_int})\ is that the integrals involved in them can be calculated in
closed-form \cite{DBesselJL}. Also, $\forall \nu \notin 
\mathbb{Z}
$, expressions in closed-form for the second and third derivatives with
respect to the order are found in \cite{BrychovNew}, but these expressions
are extraordinarily complex, above all for the third derivative.

In view of the literature commented above, the goal of this article is
two-folded. On the one hand, in Section \ref{Section: Derivatives}, we
obtain simple integral representations for the $n$-th derivatives of the
Bessel functions with respect to the order. The great advantage of these
expressions is that its numerical evaluation is quite rapid and
straightforward. As a by-product, we obtain the calculation of an integral
which does not seem to be reported in the literature.

On the other hand, in Section \ref{Section: Reflection formulas}, we derive
some reflection formulas for the first and second order derivative of $%
J_{\nu }\left( t\right) $ and $Y_{\nu }\left( t\right) $ for integral order,
from the expressions obtained in Section \ref{Section: Derivatives}. Also,
we derive reflection formulas for the $n$-th order derivative of $I_{\nu
}\left( t\right) $ and $K_{\nu }\left( t\right) $ for arbitrary real order.
As an application of the reflection formulas obtained for the first order
derivative, we extend formulas (\ref{DJ_integral})-(\ref{DK_integral}) to
negative integral orders.

Finally, we collect our conclusions in Section \ref{Section: Conclusions}.

\section{Integral representations of $n$-th order derivatives\label{Section:
Derivatives}}

In order to perform the $n$-th derivatives of Bessel and modified Bessel
functions with respect to the order, first we state the following $n$-th
derivatives, that can be proved easily by induction and using the binomial
theorem.

\begin{proposition}
The $n$-th derivative of the functions 
\begin{eqnarray*}
f_{1}\left( \nu \right) &=&\cos \left( t\sin x-\nu x\right) , \\
f_{2}\left( \nu \right) &=&\sin \left( t\sin x-\nu x\right) , \\
f_{3}\left( \nu \right) &=&e^{-\nu x}\sin \pi \nu =\Im \left( e^{\left( i\pi
-x\right) \nu }\right) , \\
f_{4}\left( \nu \right) &=&e^{-\nu x}\cos \pi \nu =\Re \left( e^{\left( i\pi
-x\right) \nu }\right) ,
\end{eqnarray*}%
with respect to the order $\nu $ are given by%
\begin{eqnarray}
\ f_{1}^{\left( n\right) }\left( \nu \right) &=&x^{n}\cos \left( t\sin x-\nu
x-n\pi /2\right) ,  \label{f1(n)} \\
\ f_{2}^{\left( n\right) }\left( \nu \right) &=&x^{n}\sin \left( t\sin x-\nu
x-n\pi /2\right) ,  \label{f2(n)} \\
f_{3}^{\left( n\right) }\left( \nu \right) &=&e^{-\nu x}\Im \left[ \left(
i\pi -x\right) ^{n}e^{i\pi \nu }\right]  \label{f3(n)} \\
&=&e^{-\nu x}\left[ p_{n}\left( x\right) \sin \pi \nu +q_{n}\left( x\right)
\cos \pi \nu \right] ,  \notag \\
f_{4}^{\left( n\right) }\left( \nu \right) &=&e^{-\nu x}\Re \left[ \left(
i\pi -x\right) ^{n}e^{i\pi \nu }\right]  \label{f4(n)} \\
&=&e^{-\nu x}\left[ p_{n}\left( x\right) \cos \pi \nu -q_{n}\left( x\right)
\sin \pi \nu \right] .  \notag
\end{eqnarray}%
where we have set the polynomials%
\begin{eqnarray}
p_{n}\left( x\right) &=&\Re \left[ \left( i\pi -x\right) ^{n}\right]
\label{pn(x)_def} \\
&=&\sum_{k=0}^{\left\lfloor n/2\right\rfloor }\binom{n}{2k}\left( -1\right)
^{n+k}\pi ^{2k}x^{n-2k},  \notag \\
q_{n}\left( x\right) &=&\Im \left[ \left( i\pi -x\right) ^{n}\right]
\label{qn(x)_def} \\
&=&\sum_{k=0}^{\left\lfloor \left( n-1\right) /2\right\rfloor }\binom{n}{2k+1%
}\left( -1\right) ^{n+k+1}\pi ^{2k+1}x^{n-2k-1}.  \notag
\end{eqnarray}
\end{proposition}

To obtain the $n$-th derivative of the Bessel function of the first kind
with respect to the order, we depart from the Schl\"{a}fli integral
representation of $J_{\nu }\left( t\right) $ \cite[Eqn. 10.9.6]{NIST},
wherein we have $\forall \Re t>0$, 
\begin{equation}
J_{\nu }\left( t\right) =\frac{1}{\pi }\int_{0}^{\pi }\cos \left( t\sin
x-\nu x\right) dx-\frac{\sin \nu \pi }{\pi }\int_{0}^{\infty }e^{-t\sinh
x-\nu x}dx.  \label{Jnu_Int}
\end{equation}

Therefore, applying (\ref{f1(n)})\ and (\ref{f3(n)}), the $n$-th derivative
of $J_{\nu }\left( t\right) $ with respect to the order is%
\begin{eqnarray}
\frac{\partial ^{n}}{\partial \nu ^{n}}J_{\nu }\left( t\right) &=&\frac{1}{%
\pi }\int_{0}^{\pi }x^{n}\cos \left( t\sin x-\nu x-\frac{\pi }{2}n\right) dx
\label{Dn_Jnu} \\
&&-\frac{1}{\pi }\int_{0}^{\infty }e^{-t\sinh x-\nu x}\Im \left[ \left( i\pi
-x\right) ^{n}e^{i\pi \nu }\right] dx  \notag \\
&=&\frac{1}{\pi }\int_{0}^{\pi }x^{n}\cos \left( t\sin x-\nu x-\frac{\pi }{2}%
n\right) dx  \label{Dn_Jnu_a} \\
&&-\frac{1}{\pi }\int_{0}^{\infty }e^{-t\sinh x-\nu x}\left[ p_{n}\left(
x\right) \sin \pi \nu +q_{n}\left( x\right) \cos \pi \nu \right] dx.  \notag
\end{eqnarray}

\begin{remark}
As a consistency test, notice that the Taylor series, 
\begin{equation}
J_{\nu }\left( t\right) =\sum_{k=0}^{\infty }\left. \frac{\partial
^{k}J_{\nu }\left( t\right) }{\partial \nu ^{k}}\right\vert _{\nu =\nu _{0}}%
\frac{\left( \nu -\nu _{0}\right) ^{k}}{k!},  \label{Taylor_J}
\end{equation}%
holds true $\forall \nu _{0}\in 
\mathbb{R}
$. Similar Taylor series hold true for $Y_{\nu }\left( t\right) $, $I_{\nu
}\left( t\right) $ and $K_{\nu }\left( t\right) $.
\end{remark}

\begin{proof}
Substituting (\ref{Dn_Jnu})\ on the LHS\ of (\ref{Taylor_J}), and reversing
the order of integration and summation, we have%
\begin{eqnarray*}
J_{\nu }\left( t\right)  &=&\frac{1}{\pi }\int_{0}^{\pi }\sum_{k=0}^{\infty }%
\frac{\left( \nu -\nu _{0}\right) ^{k}}{k!}x^{k}\cos \left( t\sin x-\nu
_{0}x-\frac{\pi }{2}k\right) dx \\
&&-\frac{1}{\pi }\int_{0}^{\infty }e^{-t\sinh x-\nu _{0}x}\sum_{k=0}^{\infty
}\frac{\left( \nu -\nu _{0}\right) ^{k}}{k!}\Im \left[ \left( i\pi -x\right)
^{k}e^{i\pi \nu _{0}}\right] dx \\
&=&\frac{1}{\pi }\int_{0}^{\pi }\Re \left[ e^{i\left( t\sin x-\nu
_{0}x\right) }\sum_{k=0}^{\infty }\frac{\left[ x\left( \nu -\nu _{0}\right) %
\right] ^{k}}{k!}e^{-i\pi k/2}\right] dx \\
&&-\frac{1}{\pi }\int_{0}^{\infty }e^{-t\sinh x-\nu _{0}x}\Im \left[ e^{i\pi
\nu _{0}}\sum_{k=0}^{\infty }\frac{\left[ \left( i\pi -x\right) \left( \nu
-\nu _{0}\right) \right] ^{k}}{k!}\right] dx \\
&=&\frac{1}{\pi }\int_{0}^{\pi }\Re \left[ e^{i\left( t\sin x-\nu
_{0}x\right) }e^{-ix\left( \nu -\nu _{0}\right) }\right] dx \\
&&-\frac{1}{\pi }\int_{0}^{\infty }e^{-t\sinh x-\nu _{0}x}\Im \left[ e^{i\pi
\nu _{0}}e^{\left( i\pi -x\right) \left( \nu -\nu _{0}\right) }\right] dx,
\end{eqnarray*}%
which is equivalent to Schl\"{a}fli integral representation (\ref{Jnu_Int}).
\end{proof}

For the Bessel function of the second kind, we obtain, from the integral
representation \cite[Eqn. 10.9.7]{NIST}, $\forall \Re t>0$, 
\begin{eqnarray}
Y_{\nu }\left( t\right) &=&\frac{1}{\pi }\int_{0}^{\pi }\sin \left( t\sin
x-\nu x\right) dx  \label{Ynu_int} \\
&&-\frac{1}{\pi }\int_{0}^{\infty }e^{-t\sinh x}\left( e^{\nu x}+e^{-\nu
x}\cos \nu \pi \right) dx,  \notag
\end{eqnarray}%
and formulas (\ref{f2(n)})\ and (\ref{f4(n)}), the following $n$-th order
derivative: 
\begin{eqnarray}
&&\frac{\partial ^{n}}{\partial \nu ^{n}}Y_{\nu }\left( t\right)
\label{Dn_Ynu} \\
&=&\frac{1}{\pi }\int_{0}^{\pi }x^{n}\sin \left( t\sin x-\nu x-\frac{\pi }{2}%
n\right) dx  \notag \\
&&-\frac{1}{\pi }\int_{0}^{\infty }e^{-t\sinh x}\left( x^{n}e^{\nu
x}+e^{-\nu x}\Re \left[ \left( i\pi -x\right) ^{n}e^{i\pi \nu }\right]
\right) dx  \notag \\
&=&\frac{1}{\pi }\int_{0}^{\pi }x^{n}\sin \left( t\sin x-\nu x-\frac{\pi }{2}%
n\right) dx  \label{Dn_Ynu_a} \\
&&-\frac{1}{\pi }\int_{0}^{\infty }e^{-t\sinh x}  \notag \\
&&\quad \left( x^{n}e^{\nu x}+e^{-\nu x}\left[ p_{n}\left( x\right) \cos \pi
\nu -q_{n}\left( x\right) \sin \pi \nu \right] \right) dx  \notag
\end{eqnarray}

For the modified Bessel function, we find in the literature the integral
representation \cite[Eqn. 10.32.4]{NIST}, $\forall \Re t>0$, 
\begin{equation}
I_{\nu }\left( t\right) =\frac{1}{\pi }\int_{0}^{\pi }e^{t\cos x}\cos \nu x\
dx-\frac{\sin \nu \pi }{\pi }\int_{0}^{\infty }e^{-t\cosh x-\nu x}dx.
\label{Inu_int}
\end{equation}

Therefore, according to (\ref{f1(n)}) with $t=0$\ and (\ref{f3(n)}), we have%
\begin{eqnarray}
&&\frac{\partial ^{n}}{\partial \nu ^{n}}I_{\nu }\left( t\right)  \notag \\
&=&\frac{1}{\pi }\int_{0}^{\pi }x^{n}e^{t\cos x}\cos \left( \nu x+\frac{\pi 
}{2}n\right) dx  \label{Dn_Inu} \\
&&-\frac{1}{\pi }\int_{0}^{\infty }e^{-t\cosh x-\nu x}\Im \left[ \left( i\pi
-x\right) ^{n}e^{i\pi \nu }\right] dx  \notag \\
&=&\frac{1}{\pi }\int_{0}^{\pi }x^{n}e^{t\cos x}\cos \left( \nu x+\frac{\pi 
}{2}n\right) dx  \label{Dn_Inu_a} \\
&&-\frac{1}{\pi }\int_{0}^{\infty }e^{-t\cosh x-\nu x}\left[ p_{n}\left(
x\right) \sin \pi \nu +q_{n}\left( x\right) \cos \pi \nu \right] dx.  \notag
\end{eqnarray}

Also, from the integral representation of the Macdonald function \cite[Eqn.
5.10.23]{Lebedev}, $\forall \Re t>0$, 
\begin{eqnarray}
K_{\nu }\left( t\right) &=&\int_{0}^{\infty }e^{-t\cosh x}\cosh \nu x\,dx
\label{Knu_int_0} \\
&=&\frac{1}{2}\int_{-\infty }^{\infty }e^{\nu x-t\cosh x}\ dx,
\label{Knu_int}
\end{eqnarray}%
we have%
\begin{eqnarray}
\frac{\partial ^{n}}{\partial \nu ^{n}}K_{\nu }\left( t\right) &=&\frac{1}{2}%
\int_{0}^{\infty }x^{n}e^{-t\cosh x}\left[ e^{\nu x}+\left( -1\right)
^{n}e^{-\nu x}\right] \,dx  \label{Dn_Knu_0} \\
&=&\frac{1}{2}\int_{-\infty }^{\infty }x^{n}e^{\nu x-t\cosh x}dx.
\label{Dn_Knu}
\end{eqnarray}

For $n=1$, the above integral (\ref{Dn_Knu})\ is calculated in \cite%
{DBesselJL}\ in closed-form, thus $\forall \nu \in 
\mathbb{R}
\diagdown \left\{ -1/2,-3/2,...\right\} $, $\Re t>0$, we have%
\begin{eqnarray}
&&\int_{-\infty }^{\infty }x\,e^{\nu x-t\cosh x}dx  \label{int_new_G} \\
&=&\nu \left[ \frac{K_{\nu }\left( z\right) }{\sqrt{\pi }}%
G_{2,4}^{3,1}\left( z^{2}\left\vert 
\begin{array}{c}
1/2,1 \\ 
0,0,\nu ,-\nu%
\end{array}%
\right. \right) -\sqrt{\pi }I_{\nu }\left( z\right) G_{2,4}^{4,0}\left(
z^{2}\left\vert 
\begin{array}{c}
1/2,1 \\ 
0,0,\nu ,-\nu%
\end{array}%
\right. \right) \right] ,  \notag
\end{eqnarray}%
where the function $G_{p,q}^{m,n}$ denotes the Meijer-$G$ function \cite[%
Eqn. 16.17.1]{NIST}. If $2\nu \notin 
\mathbb{Z}
$, the above expression is reduced in terms of generalized hypergeometric
functions $_{p}F_{q}$ \cite[Eqn. 16.2.1]{NIST} as \cite{BrychovNew},%
\begin{eqnarray}
&&\int_{-\infty }^{\infty }x\,e^{\nu x-t\cosh x}dx  \label{int_new_F} \\
&=&\pi \csc \pi \nu \left\{ \pi \cot \pi \nu \,I_{\nu }\left( z\right) - 
\left[ I_{\nu }\left( z\right) +I_{-\nu }\left( z\right) \right] 
\begin{array}{c}
\displaystyle
\\ 
\displaystyle%
\end{array}%
\right.  \notag \\
&&\left. \left[ \frac{z^{2}}{4\left( 1-\nu ^{2}\right) }\,_{3}F_{4}\left(
\left. 
\begin{array}{c}
1,1,\frac{3}{2} \\ 
2,2,2-\nu ,2+\nu%
\end{array}%
\right\vert z^{2}\right) +\log \left( \frac{z}{2}\right) -\psi \left( \nu
\right) -\frac{1}{2\nu }\right] \right\}  \notag \\
&&+\frac{1}{2}\left\{ I_{-\nu }\left( z\right) \Gamma ^{2}\left( -\nu
\right) \left( \frac{z}{2}\right) ^{2\nu }\,_{2}F_{3}\left( \left. 
\begin{array}{c}
\nu ,\frac{1}{2}+\nu \\ 
1+\nu ,1+\nu ,1+2\nu%
\end{array}%
\right\vert z^{2}\right) \right.  \notag \\
&&\quad -\left. I_{\nu }\left( z\right) \Gamma ^{2}\left( \nu \right) \left( 
\frac{z}{2}\right) ^{-2\nu }\,_{2}F_{3}\left( \left. 
\begin{array}{c}
-\nu ,\frac{1}{2}-\nu \\ 
1-\nu ,1-\nu ,1-2\nu%
\end{array}%
\right\vert z^{2}\right) \right\} .  \notag
\end{eqnarray}

It is worth noting that the numerical evaluation of the integral
representations for the $n$-th derivatives of the Bessel functions given in (%
\ref{Dn_Jnu}), (\ref{Dn_Ynu}), (\ref{Dn_Inu}), and (\ref{Dn_Knu}), is quite
efficient if we use a "double exponential"\ strategy \cite{TakahasiMori}.

\section{Reflection formulas\label{Section: Reflection formulas}}

\subsection{Bessel functions}

\begin{theorem}
$\forall t\in 
\mathbb{C}
$ and $m=0,1,\ldots $, the following reflection formula holds true:%
\begin{equation}
\left. \frac{\partial J_{\nu }\left( t\right) }{\partial \nu }\right\vert
_{\nu =m}+\left( -1\right) ^{m}\left. \frac{\partial J_{\nu }\left( t\right) 
}{\partial \nu }\right\vert _{\nu =-m}=\pi \,Y_{m}\left( t\right) .
\label{Reflection_J}
\end{equation}
\end{theorem}

\begin{proof}
From the integral representation (\ref{Dn_Jnu}), we have%
\begin{eqnarray}
\left. \frac{\partial J_{\nu }\left( t\right) }{\partial \nu }\right\vert
_{\nu =\pm m} &=&\frac{1}{\pi }\int_{0}^{\pi }x\sin \left( t\sin x\mp
mx\right) dx  \label{DJ_m} \\
&&+\left( -1\right) ^{m+1}\int_{0}^{\infty }e^{-t\sinh x}e^{\mp mx}dx. 
\notag
\end{eqnarray}

On the one hand, consider $m=2k$, thus, according to (\ref{DJ_m}), the LHS
of (\ref{Reflection_J})\ becomes%
\begin{eqnarray}
\left. \frac{\partial J_{\nu }\left( t\right) }{\partial \nu }\right\vert
_{\nu =2k}+\left. \frac{\partial J_{\nu }\left( t\right) }{\partial \nu }%
\right\vert _{\nu =-2k} &=&\frac{2}{\pi }\int_{0}^{\pi }x\sin \left( t\sin
x\right) \cos \left( 2kx\right) dx  \label{DJ_2k} \\
&&-\int_{0}^{\infty }e^{-t\sinh x}\left( e^{2kx}+e^{-2kx}\right) dx.  \notag
\end{eqnarray}%
To calculate the first integral on the RHS of (\ref{DJ_2k}), perform the
substitution $\xi =x-\pi /2$, eliminate the term that vanishes by parity,
and undo the change of variables to obtain%
\begin{equation}
\frac{2}{\pi }\int_{0}^{\pi }x\sin \left( t\sin x\right) \cos \left(
2kx\right) dx=\int_{0}^{\pi }\sin \left( t\sin x\right) \cos \left(
2kx\right) dx.  \label{Int_J_2k}
\end{equation}%
Applying the trigonometric identity $\sin \left( \alpha +\beta \right) =\sin
\alpha \cos \beta +\sin \beta \cos \alpha $, rewrite (\ref{Int_J_2k})\ as%
\begin{eqnarray}
\frac{2}{\pi }\int_{0}^{\pi }x\sin \left( t\sin x\right) \cos \left(
2kx\right) dx &=&\int_{0}^{\pi }\sin \left( t\sin x+2kx\right) dx
\label{Int_J_2k_a} \\
&&-\int_{0}^{\pi }\cos \left( t\sin x\right) \sin \left( 2kx\right) dx. 
\notag
\end{eqnarray}%
The second integral on the RHS\ of (\ref{Int_J_2k_a})\ vanishes by parity
performing the substitution $\xi =x-\pi /2$. Therefore, (\ref{DJ_2k})\
becomes%
\begin{eqnarray}
\left. \frac{\partial J_{\nu }\left( t\right) }{\partial \nu }\right\vert
_{\nu =2k}+\left. \frac{\partial J_{\nu }\left( t\right) }{\partial \nu }%
\right\vert _{\nu =-2k} &=&\int_{0}^{\pi }\sin \left( t\sin x+2kx\right) dx
\label{Int_J_2k_b} \\
&&-\int_{0}^{\infty }e^{-t\sinh x}\left( e^{2kx}+e^{-2kx}\right) dx.  \notag
\end{eqnarray}%
According to the integral representation (\ref{Ynu_int}) and the property 
\cite[Eqn. 10.4.1]{NIST}%
\begin{equation}
Y_{-m}\left( t\right) =\left( -1\right) ^{m}Y_{m}\left( t\right) ,
\label{Ym_Y-m}
\end{equation}%
we rewrite (\ref{Int_J_2k_b})\ as (\ref{Reflection_J}) for $m=2k$. This
completes the proof for $m=2k$.

On the other hand, consider $m=2k+1$, thus, according to (\ref{DJ_m}), the
LHS\ of (\ref{Reflection_J})\ becomes%
\begin{eqnarray}
&&\left. \frac{\partial J_{\nu }\left( t\right) }{\partial \nu }\right\vert
_{\nu =2k+1}-\left. \frac{\partial J_{\nu }\left( t\right) }{\partial \nu }%
\right\vert _{\nu =-2k-1}  \label{DJ_2k+1} \\
&=&\frac{-2}{\pi }\int_{0}^{\pi }x\cos \left( t\sin x\right) \sin \left(
\left( 2k+1\right) x\right) dx  \notag \\
&&+\int_{0}^{\infty }e^{-t\sinh x}\left( e^{\left( 2k+1\right) x}-e^{-\left(
2k+1\right) x}\right) dx.  \notag
\end{eqnarray}%
Following the same steps as in (\ref{Int_J_2k})-(\ref{Int_J_2k_b}), we
calculate the first integral on the RHS of (\ref{DJ_2k+1}) as,%
\begin{equation}
\frac{-2}{\pi }\int_{0}^{\pi }x\cos \left( t\sin x\right) \sin \left( \left(
2k+1\right) x\right) dx=-\int_{0}^{\pi }\sin \left( t\sin x+\left(
2k+1\right) x\right) dx.  \label{Int_2k+1}
\end{equation}%
Therefore, inserting (\ref{Int_2k+1})\ in (\ref{DJ_2k+1})\ and taking into
account the integral representation (\ref{Ynu_int}) and the property (\ref%
{Ym_Y-m}), we rewrite (\ref{Int_2k+1})\ as%
\begin{equation*}
\left. \frac{\partial J_{\nu }\left( t\right) }{\partial \nu }\right\vert
_{\nu =2k+1}-\left. \frac{\partial J_{\nu }\left( t\right) }{\partial \nu }%
\right\vert _{\nu =-2k-1}=\pi \,Y_{2k+1}\left( t\right) ,
\end{equation*}%
which completes the proof.
\end{proof}

\begin{corollary}
We can extend the formula given in (\ref{DJ_integral})\ to negative integral
orders with the aid of the reflection formula (\ref{Reflection_J}),
resulting in%
\begin{equation}
\left. \frac{\partial J_{\nu }\left( t\right) }{\partial \nu }\right\vert
_{\nu =\pm m}=\left( \pm 1\right) ^{m}\left[ \frac{\pi }{2}Y_{m}\left(
t\right) \pm \frac{m!}{2}\sum_{k=0}^{m-1}\frac{J_{k}\left( t\right) }{%
k!\left( m-k\right) }\left( \frac{t}{2}\right) ^{k-m}\right] .
\label{Extension_J}
\end{equation}
\end{corollary}

\begin{theorem}
$\forall t\in 
\mathbb{C}
$ and $m=0,1,\ldots $, the following reflection formula holds true:%
\begin{equation}
\left. \frac{\partial ^{2}J_{\nu }\left( t\right) }{\partial \nu ^{2}}%
\right\vert _{\nu =m}+\left( -1\right) ^{m+1}\left. \frac{\partial
^{2}J_{\nu }\left( t\right) }{\partial \nu ^{2}}\right\vert _{\nu =-m}=2\pi
\left. \frac{\partial Y_{\nu }\left( t\right) }{\partial \nu }\right\vert
_{\nu =m}+\pi ^{2}J_{m}\left( t\right) .  \label{Reflection_J2}
\end{equation}
\end{theorem}

\begin{proof}
From the integral representation (\ref{Dn_Jnu}) for $n=2$, we have 
\begin{eqnarray*}
\left. \frac{\partial ^{2}J_{\nu }\left( t\right) }{\partial \nu ^{2}}%
\right\vert _{\nu =\pm m} &=&\frac{-1}{\pi }\int_{0}^{\pi }x^{2}\cos \left(
t\sin x\mp mx\right) dx \\
&&+2\left( -1\right) ^{m}\int_{0}^{\infty }e^{-t\sinh x}x\,e^{\mp mx}dx.
\end{eqnarray*}%
Therefore, 
\begin{eqnarray}
&&\left. \frac{\partial ^{2}J_{\nu }\left( t\right) }{\partial \nu ^{2}}%
\right\vert _{\nu =m}+\left( -1\right) ^{m+1}\left. \frac{\partial
^{2}J_{\nu }\left( t\right) }{\partial \nu ^{2}}\right\vert _{\nu =-m}
\label{D2J_a} \\
&=&\frac{-1}{\pi }\int_{0}^{\pi }x^{2}\left[ \cos \left( t\sin x-mx\right)
+\left( -1\right) ^{m+1}\cos \left( t\sin x+mx\right) \right] dx  \notag \\
&&+2\int_{0}^{\infty }e^{-t\sinh x}x\left[ \left( -1\right)
^{m}e^{-mx}-e^{mx}\right] dx.  \notag
\end{eqnarray}%
On the one hand, calculate the first integral on the RHS of (\ref{D2J_a})
for $m=2k$,%
\begin{eqnarray*}
&&\frac{1}{\pi }\int_{0}^{\pi }x^{2}\left[ \cos \left( t\sin x+2kx\right)
-\cos \left( t\sin x-2kx\right) \right] dx \\
&=&\frac{-2}{\pi }\int_{0}^{\pi }x^{2}\sin \left( t\sin x\right) \sin 2kx\
dx.
\end{eqnarray*}%
Performing the change of variables $\xi =x-\pi /2$ and cancelling the
corresponding terms by parity, we arrive at%
\begin{equation}
2\left( -1\right) ^{k}\int_{-\pi /2}^{\pi /2}\xi \sin \left( t\cos \xi
\right) \sin \left( -2k\xi \right) d\xi ,  \label{D2J_b}
\end{equation}%
Applying the trigonometric identity $\cos \left( \alpha +\beta \right) =\cos
\alpha \cos \beta -\sin \alpha \sin \beta $, eliminating one of the
resulting integrals by parity, and undoing the substitution performed, we
rewrite (\ref{D2J_b})\ as 
\begin{eqnarray}
&&-2\int_{0}^{\pi }\left( x-\frac{\pi }{2}\right) \cos \left( t\sin
x-2kx\right) dx  \label{D2J_c} \\
&=&-2\int_{0}^{\pi }x\cos \left( t\sin x-2kx\right) dx+\pi
^{2}\,J_{2k}\left( t\right) ,  \notag
\end{eqnarray}%
where we have applied the integral representation (\ref{Jnu_Int}).

On the other hand, calculate the first integral on the RHS of (\ref{D2J_a})
for $m=2k+1$,%
\begin{eqnarray*}
&&\frac{-1}{\pi }\int_{0}^{\pi }x^{2}\left[ \cos \left( t\sin x-\left(
2k+1\right) x\right) +\cos \left( t\sin x+\left( 2k+1\right) x\right) \right]
dx \\
&=&\frac{-2}{\pi }\int_{0}^{\pi }x^{2}\cos \left( t\sin x\right) \cos \left(
\left( 2k+1\right) x\right) dx.
\end{eqnarray*}%
Following the same steps as in (\ref{D2J_b})-(\ref{D2J_c}), but applying the
trigonometric identity $\sin \left( \alpha +\beta \right) =\sin \alpha \cos
\beta +\sin \beta \cos \alpha $, we rewrite the first integral on the LHS of
(\ref{D2J_a})\ as%
\begin{equation}
-2\int_{0}^{\pi }x\cos \left( t\sin x-\left( 2k+1\right) x\right) dx+\pi
^{2}J_{2k+1}\left( t\right) .  \label{D2J_d}
\end{equation}%
From (\ref{D2J_c})\ and (\ref{D2J_d}), Eqn. (\ref{D2J_a})\ becomes 
\begin{eqnarray}
&&\left. \frac{\partial ^{2}J_{\nu }\left( t\right) }{\partial \nu ^{2}}%
\right\vert _{\nu =m}+\left( -1\right) ^{m+1}\left. \frac{\partial
^{2}J_{\nu }\left( t\right) }{\partial \nu ^{2}}\right\vert _{\nu =-m}
\label{D2J_final} \\
&=&-2\int_{0}^{\pi }x\cos \left( t\sin x-mx\right) dx+\pi ^{2}J_{m}\left(
t\right)  \notag \\
&&+2\int_{0}^{\infty }e^{-t\sinh x}x\left[ \left( -1\right)
^{m}e^{-mx}-e^{mx}\right] dx.  \notag
\end{eqnarray}%
Finally, take into account the integral representation (\ref{Dn_Ynu})\ for $%
n=1$, to express (\ref{D2J_final})\ as (\ref{Reflection_J2}), as we wanted
to prove.
\end{proof}

\begin{corollary}
Taking into account (\ref{DY_integral}), Eqn. (\ref{Reflection_J2})\ becomes 
\begin{equation}
\left. \frac{\partial ^{2}J_{\nu }\left( t\right) }{\partial \nu ^{2}}%
\right\vert _{\nu =m}+\left( -1\right) ^{m+1}\left. \frac{\partial
^{2}J_{\nu }\left( t\right) }{\partial \nu ^{2}}\right\vert _{\nu =-m}=\pi
\,m!\sum_{k=0}^{m-1}\frac{Y_{k}\left( t\right) }{k!\left( m-k\right) }\left( 
\frac{t}{2}\right) ^{k-m}.  \label{Reflection_J2_closed-form}
\end{equation}
\end{corollary}

\begin{theorem}
$\forall t\in 
\mathbb{C}
$ and $m=0,1,\ldots $, the following reflection formula holds true:%
\begin{equation}
\left. \frac{\partial Y_{\nu }\left( t\right) }{\partial \nu }\right\vert
_{\nu =m}+\left( -1\right) ^{m}\left. \frac{\partial Y_{\nu }\left( t\right) 
}{\partial \nu }\right\vert _{\nu =-m}=-\pi \,J_{m}\left( t\right) .
\label{Reflection_Y}
\end{equation}
\end{theorem}

\begin{proof}
From the integral representation (\ref{Dn_Ynu}), we have%
\begin{eqnarray*}
\left. \frac{\partial Y_{\nu }\left( t\right) }{\partial \nu }\right\vert
_{\nu =\pm m} &=&\frac{-1}{\pi }\int_{0}^{\pi }x\cos \left( t\sin x\mp
mx\right) dx \\
&&-\frac{1}{\pi }\int_{0}^{\infty }x\,e^{-t\sinh x}\left( e^{\pm mx}+\left(
-1\right) ^{m+1}e^{\mp mx}\right) dx.
\end{eqnarray*}%
Therefore,%
\begin{eqnarray}
&&\left. \frac{\partial Y_{\nu }\left( t\right) }{\partial \nu }\right\vert
_{\nu =m}+\left( -1\right) ^{m}\left. \frac{\partial Y_{\nu }\left( t\right) 
}{\partial \nu }\right\vert _{\nu =-m}  \label{DYm} \\
&=&\frac{-1}{\pi }\int_{0}^{\pi }x\left[ \cos \left( t\sin x-mx\right)
+\left( -1\right) ^{m}\cos \left( t\sin x+mx\right) \right] dx.  \notag
\end{eqnarray}

On the one hand, consider $m=2k$, thus (\ref{DYm})\ becomes%
\begin{eqnarray}
&&\left. \frac{\partial Y_{\nu }\left( t\right) }{\partial \nu }\right\vert
_{\nu =2k}+\left. \frac{\partial Y_{\nu }\left( t\right) }{\partial \nu }%
\right\vert _{\nu =-2k}  \label{DY_2k_0} \\
&=&\frac{-2}{\pi }\int_{0}^{\pi }x\cos \left( t\sin x\right) \cos \left(
2kx\right) dx.  \notag
\end{eqnarray}%
Applying the same steps as in (\ref{Int_J_2k})-(\ref{Int_J_2k_b}), but
considering the trigonometric identity $\cos \left( \alpha +\beta \right)
=\cos \alpha \cos \beta -\sin \alpha \sin \beta $, we arrive at%
\begin{eqnarray}
&&\left. \frac{\partial Y_{\nu }\left( t\right) }{\partial \nu }\right\vert
_{\nu =2k}+\left. \frac{\partial Y_{\nu }\left( t\right) }{\partial \nu }%
\right\vert _{\nu =-2k}  \label{DY_2k} \\
&=&-\int_{0}^{\pi }\cos \left( t\sin x+2kx\right) dx.  \notag
\end{eqnarray}%
According to the integral representation (\ref{Jnu_Int})\ and the property 
\cite[Eqn. 10.4.1]{NIST}%
\begin{equation}
J_{-m}\left( t\right) =\left( -1\right) ^{m}J_{m}\left( t\right) ,
\label{Jm_J-m}
\end{equation}%
we rewrite (\ref{DY_2k}) as (\ref{Reflection_Y})\ for $m=2k$. This completes
the proof for $m=2k$.

On the other hand, consider $m=2k+1$, thus (\ref{DYm})\ becomes%
\begin{eqnarray*}
&&\left. \frac{\partial Y_{\nu }\left( t\right) }{\partial \nu }\right\vert
_{\nu =2k+1}-\left. \frac{\partial Y_{\nu }\left( t\right) }{\partial \nu }%
\right\vert _{\nu =-2k-1} \\
&=&\frac{-2}{\pi }\int_{0}^{\pi }x\sin \left( t\sin x\right) \sin \left(
\left( 2k+1\right) x\right) dx.
\end{eqnarray*}%
Similar to the derivation given in (\ref{DY_2k_0})-(\ref{DY_2k}), we
calculate the above integral as 
\begin{eqnarray*}
&&\left. \frac{\partial Y_{\nu }\left( t\right) }{\partial \nu }\right\vert
_{\nu =2k+1}-\left. \frac{\partial Y_{\nu }\left( t\right) }{\partial \nu }%
\right\vert _{\nu =-2k-1} \\
&=&\int_{0}^{\pi }\cos \left( t\sin x+\left( 2k+1\right) x\right) dx,
\end{eqnarray*}%
thereby, according to the integral representation (\ref{Jnu_Int})\ and the
property (\ref{Jm_J-m}), we finally arrive at (\ref{Reflection_Y}) for $%
m=2k+1$. This completes the proof.
\end{proof}

\begin{corollary}
We extend the formula given in (\ref{DY_integral})\ to negative integral
orders with the aid of the reflection formula (\ref{Reflection_Y}),
resulting in%
\begin{equation}
\left. \frac{\partial Y_{\nu }\left( t\right) }{\partial \nu }\right\vert
_{\nu =\pm m}=\left( \pm 1\right) ^{m}\left[ -\frac{\pi }{2}J_{m}\left(
t\right) \pm \frac{m!}{2}\sum_{k=0}^{m-1}\frac{Y_{k}\left( t\right) }{%
k!\left( m-k\right) }\left( \frac{t}{2}\right) ^{k-m}\right] .
\label{Extenxion_Y}
\end{equation}
\end{corollary}

\begin{theorem}
$\forall t\in 
\mathbb{C}
$ and $m=0,1,\ldots $, the following reflection formula holds true:%
\begin{equation}
\left. \frac{\partial ^{2}Y_{\nu }\left( t\right) }{\partial \nu ^{2}}%
\right\vert _{\nu =m}+\left( -1\right) ^{m+1}\left. \frac{\partial
^{2}Y_{\nu }\left( t\right) }{\partial \nu ^{2}}\right\vert _{\nu =-m}=-2\pi
\left. \frac{\partial J_{\nu }\left( t\right) }{\partial \nu }\right\vert
_{\nu =m}+\pi ^{2}Y_{m}\left( t\right) .  \label{Reflection_Y2}
\end{equation}
\end{theorem}

\begin{proof}
From the integral representation (\ref{Dn_Ynu})\ for $n=2$, we have%
\begin{eqnarray*}
&&\left. \frac{\partial ^{2}Y_{\nu }\left( t\right) }{\partial \nu ^{2}}%
\right\vert _{\nu =\pm m} \\
&=&\frac{-1}{\pi }\int_{0}^{\pi }x^{2}\sin \left( t\sin x\mp mx\right) dx \\
&&-\frac{1}{\pi }\int_{0}^{\infty }e^{-t\sinh x}\left[ x^{2}\left( e^{\pm
mx}+\left( -1\right) ^{m}e^{\mp mx}\right) +\left( -1\right) ^{m+1}\pi
^{2}e^{\mp mx}\right] dx,
\end{eqnarray*}%
hence%
\begin{eqnarray}
&&\left. \frac{\partial ^{2}Y_{\nu }\left( t\right) }{\partial \nu ^{2}}%
\right\vert _{\nu =m}+\left( -1\right) ^{m+1}\left. \frac{\partial
^{2}Y_{\nu }\left( t\right) }{\partial \nu ^{2}}\right\vert _{\nu =-m}
\label{D2Y_a} \\
&=&\frac{-1}{\pi }\int_{0}^{\pi }x^{2}\left[ \sin \left( t\sin x-mx\right)
+\left( -1\right) ^{m+1}\sin \left( t\sin x+mx\right) \right] dx  \notag \\
&&-\pi \int_{0}^{\infty }e^{-t\sinh x}\left[ \left( -1\right)
^{m+1}e^{-mx}+e^{mx}\right] dx.  \notag
\end{eqnarray}%
Let us calculate the first integral on the RHS of (\ref{D2Y_a}). First,
consider $m=2k$, thus%
\begin{eqnarray*}
&&\frac{-1}{\pi }\int_{0}^{\pi }x^{2}\left[ \sin \left( t\sin x-2kx\right)
-\sin \left( t\sin x+2kx\right) \right] dx \\
&=&\frac{2}{\pi }\int_{0}^{\pi }x^{2}\cos \left( t\sin x\right) \sin 2kx\,dx.
\end{eqnarray*}%
Performing the substitution $\xi =x-\pi /2$ and cancelling the corresponding
terms by parity, we have%
\begin{equation}
-2\left( -1\right) ^{k}\int_{-\pi /2}^{\pi /2}\xi \cos \left( t\cos \xi
\right) \sin \left( -2k\xi \right) d\xi .  \label{D2Y_b}
\end{equation}%
Rewrite (\ref{D2Y_b})\ with the trigonometric identity $\sin \left( \alpha
+\beta \right) =\sin \alpha \cos \beta +\sin \beta \cos \alpha $, eliminate
one of the resulting integrals by parity, and undo the substitution
performed, arriving at%
\begin{eqnarray}
&&\frac{-1}{\pi }\int_{0}^{\pi }x^{2}\left[ \sin \left( t\sin x-2kx\right)
-\sin \left( t\sin x+2kx\right) \right] dx  \label{D2Y_c} \\
&=&-2\int_{0}^{\pi }x\sin \left( t\sin x-2kx\right) dx+\pi \int_{0}^{\pi
}\sin \left( t\sin x-2kx\right) dx.  \notag
\end{eqnarray}

Second, consider the case $m=2k+1$, thus the first integral on the RHS\ of (%
\ref{D2Y_a})\ becomes%
\begin{eqnarray*}
&&\frac{-1}{\pi }\int_{0}^{\pi }x^{2}\left[ \sin \left( t\sin x-\left(
2k+1\right) x\right) +\sin \left( t\sin x+\left( 2k+1\right) x\right) \right]
dx \\
&=&\frac{2}{\pi }\int_{0}^{\pi }x^{2}\sin \left( t\sin x\right) \cos \left(
\left( 2k+1\right) x\right) dx.
\end{eqnarray*}%
Following the same steps as before, but applying the trigonometric identity $%
\cos \left( \alpha +\beta \right) =\cos \alpha \cos \beta -\sin \alpha \sin
\beta $, we arrive at 
\begin{equation}
-2\int_{0}^{\pi }x\sin \left( t\sin x-\left( 2k+1\right) x\right) dx+\pi
\int_{0}^{\pi }\sin \left( t\sin x-\left( 2k+1\right) x\right) dx.
\label{D2Y_d}
\end{equation}%
Therefore, taking into account (\ref{D2Y_c}) and (\ref{D2Y_d}),%
\begin{eqnarray}
&&\left. \frac{\partial ^{2}Y_{\nu }\left( t\right) }{\partial \nu ^{2}}%
\right\vert _{\nu =m}+\left( -1\right) ^{m+1}\left. \frac{\partial
^{2}Y_{\nu }\left( t\right) }{\partial \nu ^{2}}\right\vert _{\nu =-m}
\label{D2Y_final} \\
&=&-2\int_{0}^{\pi }x\sin \left( t\sin x-mx\right) dx+\pi \int_{0}^{\pi
}\sin \left( t\sin x-mx\right) dx  \notag \\
&&-\pi \int_{0}^{\infty }e^{-t\sinh x}\left[ \left( -1\right)
^{m+1}e^{-mx}+e^{mx}\right] dx.  \notag
\end{eqnarray}%
Finally, notice that from the integral representation (\ref{Dn_Jnu})\ for $%
n=1$, and the integral representation (\ref{Ynu_int}), we see that the RHS\
of both (\ref{Reflection_Y2})\ and (\ref{D2Y_final})\ are equal, completing
thereby the proof.
\end{proof}

\begin{corollary}
Taking into account (\ref{DJ_integral}), Eqn. (\ref{Reflection_Y2})\ becomes 
\begin{equation}
\left. \frac{\partial ^{2}Y_{\nu }\left( t\right) }{\partial \nu ^{2}}%
\right\vert _{\nu =m}+\left( -1\right) ^{m+1}\left. \frac{\partial
^{2}Y_{\nu }\left( t\right) }{\partial \nu ^{2}}\right\vert _{\nu =-m}=-\pi
\,m!\sum_{k=0}^{m-1}\frac{J_{k}\left( t\right) }{k!\left( m-k\right) }\left( 
\frac{t}{2}\right) ^{k-m}.  \label{Reflection_Y2_closed-form}
\end{equation}
\end{corollary}

\subsection{Modified Bessel functions}

\begin{theorem}
$\forall t\in 
\mathbb{C}
$ and $\mu \in 
\mathbb{R}
$, the following reflection formula holds true:%
\begin{eqnarray}
&&\left. \frac{\partial ^{n}I_{\nu }\left( t\right) }{\partial \nu ^{n}}%
\right\vert _{\nu =\mu }+\left( -1\right) ^{n+1}\left. \frac{\partial
^{n}I_{\nu }\left( t\right) }{\partial \nu ^{n}}\right\vert _{\nu =-\mu }
\label{Reflection_I_general} \\
&=&\frac{-2}{\pi }\sin \pi \mu \sum_{k=0}^{\left\lfloor n/2\right\rfloor }%
\binom{n}{2k}\left( -\pi ^{2}\right) ^{k}\frac{\partial ^{n-2k}K_{\mu
}\left( t\right) }{\partial \mu ^{n-2k}}  \notag \\
&&-2\cos \pi \mu \sum_{k=0}^{\left\lfloor \left( n-1\right) /2\right\rfloor }%
\binom{n}{2k+1}\left( -\pi ^{2}\right) ^{k}\frac{\partial ^{n-2k-1}K_{\mu
}\left( t\right) }{\partial \mu ^{n-2k-1}}.  \notag
\end{eqnarray}
\end{theorem}

\begin{proof}
Applying the integral representation (\ref{Dn_Inu_a}), we arrive at 
\begin{eqnarray*}
&&\left. \frac{\partial ^{n}I_{\nu }\left( t\right) }{\partial \nu ^{n}}%
\right\vert _{\nu =\mu }+\left( -1\right) ^{n+1}\left. \frac{\partial
^{n}I_{\nu }\left( t\right) }{\partial \nu ^{n}}\right\vert _{\nu =-\mu } \\
&=&-\frac{1}{\pi }\int_{0}^{\infty }e^{-t\cosh x}\left\{ p_{n}\left(
x\right) \sin \pi \mu \left[ e^{-\mu x}+\left( -1\right) ^{n}e^{\mu x}\right]
\right. \\
&&\quad +\left. q_{n}\left( x\right) \sin \pi \mu \left[ e^{-\mu x}+\left(
-1\right) ^{n+1}e^{\mu x}\right] \right\} dx,
\end{eqnarray*}%
wherein the first integral of (\ref{Dn_Inu_a}) vanishes. Taking into account
the definitions of the polynomials $p_{n}$ and $q_{n}$ given in (\ref%
{pn(x)_def})\ and (\ref{qn(x)_def}), 
\begin{eqnarray}
&&\left. \frac{\partial ^{n}I_{\nu }\left( t\right) }{\partial \nu ^{n}}%
\right\vert _{\nu =\mu }+\left( -1\right) ^{n+1}\left. \frac{\partial
^{n}I_{\nu }\left( t\right) }{\partial \nu ^{n}}\right\vert _{\nu =-\mu }
\label{DIn_general} \\
&=&\frac{-1}{\pi }\sin \pi \mu \sum_{k=0}^{\left\lfloor n/2\right\rfloor }%
\binom{n}{2k}\left( -\pi ^{2}\right) ^{k}  \notag \\
&&\int_{0}^{\infty }e^{-t\cosh x}x^{n-2k}\left[ e^{-\mu x}\left( -1\right)
^{n}+e^{\mu x}\right] dx  \notag \\
&&-\cos \pi \mu \sum_{k=0}^{\left\lfloor \left( n-1\right) /2\right\rfloor }%
\binom{n}{2k+1}\left( -\pi ^{2}\right) ^{k}  \notag \\
&&\int_{0}^{\infty }e^{-t\cosh x}x^{n-2k-1}\left[ e^{-\mu x}\left( -1\right)
^{n+1}+e^{\mu x}\right] dx.  \notag
\end{eqnarray}%
Finally, taking into account the integral representation (\ref{Dn_Knu_0}),
we rewrite (\ref{DIn_general}) as (\ref{Reflection_I_general}), as we wanted
to prove.
\end{proof}

From the general expression (\ref{Reflection_I_general}), we obtain some
interesting particular cases. Thereby, on the one hand, for\ $\mu
=m=0,1,\ldots $,%
\begin{eqnarray*}
&&\left. \frac{\partial ^{n}I_{\nu }\left( t\right) }{\partial \nu ^{n}}%
\right\vert _{\nu =m}+\left( -1\right) ^{n+1}\left. \frac{\partial
^{n}I_{\nu }\left( t\right) }{\partial \nu ^{n}}\right\vert _{\nu =-m} \\
&=&2\left( -1\right) ^{m+1}\sum_{k=0}^{\left\lfloor \left( n-1\right)
/2\right\rfloor }\binom{n}{2k+1}\left( -\pi ^{2}\right) ^{k}\left. \frac{%
\partial ^{n-2k-1}K_{\nu }\left( t\right) }{\partial \nu ^{n-2k-1}}%
\right\vert _{\nu =m},
\end{eqnarray*}%
thus $\forall n=0,1,2$, we have respectively%
\begin{equation}
I_{m}\left( t\right) -I_{-m}\left( t\right) =0,  \label{Reflection_I0}
\end{equation}%
which is a well-known expression given in the the literature \cite[Eqn.
10.27.1]{NIST},%
\begin{equation}
\left. \frac{\partial I_{\nu }\left( t\right) }{\partial \nu }\right\vert
_{\nu =m}+\left. \frac{\partial I_{\nu }\left( t\right) }{\partial \nu }%
\right\vert _{\nu =-m}=2\left( -1\right) ^{m+1}\,K_{m}\left( t\right) ,
\label{Reflection_I}
\end{equation}%
and 
\begin{equation}
\left. \frac{\partial ^{2}I_{\nu }\left( t\right) }{\partial \nu ^{2}}%
\right\vert _{\nu =m}-\left. \frac{\partial ^{2}I_{\nu }\left( t\right) }{%
\partial \nu ^{2}}\right\vert _{\nu =-m}=4\left( -1\right) ^{m+1}\left. 
\frac{\partial K_{\nu }\left( t\right) }{\partial \nu }\right\vert _{\nu =m}.
\label{Reflection_I2}
\end{equation}%
Therefore, we state the following two results.

\begin{corollary}
We extend the formula given in (\ref{DI_integral})\ to negative integral
orders with the aid of the reflection formula (\ref{Reflection_I}),
resulting in%
\begin{equation}
\left. \frac{\partial I_{\nu }\left( t\right) }{\partial \nu }\right\vert
_{\nu =\pm m}=\left( -1\right) ^{m}\left[ -K_{m}\left( t\right) \pm \frac{m!%
}{2}\sum_{k=0}^{m-1}\frac{\left( -1\right) ^{k}I_{k}\left( t\right) }{%
k!\left( m-k\right) }\left( \frac{t}{2}\right) ^{k-m}\right] .
\label{Extension_I}
\end{equation}
\end{corollary}

\begin{corollary}
Taking into account (\ref{DK_integral}), Eqn. (\ref{Reflection_I2})\ becomes%
\begin{equation}
\left. \frac{\partial ^{2}I_{\nu }\left( t\right) }{\partial \nu ^{2}}%
\right\vert _{\nu =m}-\left. \frac{\partial ^{2}I_{\nu }\left( t\right) }{%
\partial \nu ^{2}}\right\vert _{\nu =-m}=2\left( -1\right)
^{m+1}m!\sum_{k=0}^{m-1}\frac{K_{k}\left( t\right) }{k!\left( m-k\right) }%
\left( \frac{t}{2}\right) ^{k-m}  \label{Reflection_I2_closed-form}
\end{equation}
\end{corollary}

On the other hand, taking in (\ref{Reflection_I_general})\ $\mu =m+\frac{1}{2%
}$, we have%
\begin{eqnarray*}
&&\left. \frac{\partial ^{n}I_{\nu }\left( t\right) }{\partial \nu ^{n}}%
\right\vert _{\nu =m+1/2}+\left( -1\right) ^{n+1}\left. \frac{\partial
^{n}I_{\nu }\left( t\right) }{\partial \nu ^{n}}\right\vert _{\nu =-m-1/2} \\
&=&\frac{2}{\pi }\left( -1\right) ^{m+1}\sum_{k=0}^{\left\lfloor
n/2\right\rfloor }\binom{n}{2k}\left( -\pi ^{2}\right) ^{k}\left. \frac{%
\partial ^{n-2k}K_{\nu }\left( t\right) }{\partial \nu ^{n-2k}}\right\vert
_{\nu =m+1/2},
\end{eqnarray*}%
thus $\forall n=0,1$, we have respectively%
\begin{equation*}
I_{m+1/2}\left( t\right) -I_{-m-1/2}\left( t\right) =\frac{2}{\pi }\left(
-1\right) ^{m+1}K_{m+1/2}\left( t\right) ,
\end{equation*}%
which is given in the literature \cite[Eqn. 10.47.11]{NIST}, and%
\begin{equation}
\left. \frac{\partial I_{\nu }\left( t\right) }{\partial \nu }\right\vert
_{\nu =m+1/2}+\left. \frac{\partial I_{\nu }\left( t\right) }{\partial \nu }%
\right\vert _{\nu =-m-1/2}=\frac{2}{\pi }\left( -1\right) ^{m+1}\left. \frac{%
\partial K_{\nu }\left( t\right) }{\partial \nu }\right\vert _{\nu =m+1/2},
\label{Reflection_I_half}
\end{equation}%
which is also fulfilled for $m=0$ by the formulas found in \cite[Eqn.
10.38.6\&7]{NIST}.

\begin{theorem}
$\forall t\in 
\mathbb{C}
$ and $m=0,1,\ldots $, the following reflection formula holds true:%
\begin{equation}
\left. \frac{\partial ^{n}K_{\nu }\left( t\right) }{\partial \nu ^{n}}%
\right\vert _{\nu =\mu }+\left( -1\right) ^{n+1}\left. \frac{\partial
^{n}K_{\nu }\left( t\right) }{\partial \nu ^{n}}\right\vert _{\nu =-\mu }=0.
\label{Reflection_K}
\end{equation}
\end{theorem}

\begin{proof}
First, consider the case $n=2k$ and take into account the integral
representation (\ref{DKnu_int}), thereby (\ref{Reflection_K})\ reads as%
\begin{equation*}
\left. \frac{\partial ^{2k}K_{\nu }\left( t\right) }{\partial \nu ^{2k}}%
\right\vert _{\nu =\mu }-\left. \frac{\partial ^{2k}K_{\nu }\left( t\right) 
}{\partial \nu ^{2k}}\right\vert _{\nu =-\mu }=\int_{-\infty }^{\infty
}x^{2k}e^{-t\cosh x}\sinh \left( \mu x\right) dx=0,
\end{equation*}%
which vanishes by parity. Second, consider the case $n=2k+1$ and the
integral representation (\ref{DKnu_int}), thus (\ref{Reflection_K})\ becomes%
\begin{equation*}
\left. \frac{\partial ^{2k+1}K_{\nu }\left( t\right) }{\partial \nu ^{2k+1}}%
\right\vert _{\nu =\mu }+\left. \frac{\partial ^{2k+1}K_{\nu }\left(
t\right) }{\partial \nu ^{2k+1}}\right\vert _{\nu =-\mu }=\int_{-\infty
}^{\infty }x^{2k+1}e^{-t\cosh x}\cosh \left( \mu x\right) dx=0,
\end{equation*}%
which is also null by parity. This completes the proof.
\end{proof}

\begin{corollary}
We extend the formula given in (\ref{DK_integral})\ to negative integral
orders, taking $n=1$ and $\mu =m=0,1,\ldots $ in (\ref{Reflection_K}),
resulting in%
\begin{equation}
\left. \frac{\partial K_{\nu }\left( t\right) }{\partial \nu }\right\vert
_{\nu =\pm m}=\pm \frac{m!}{2}\sum_{k=0}^{m-1}\frac{K_{k}\left( t\right) }{%
k!\left( m-k\right) }\left( \frac{t}{2}\right) ^{k-m}.  \label{Extension_K}
\end{equation}
\end{corollary}

\begin{corollary}
From (\ref{Dn_Knu}) with $n=1$\ and (\ref{Extension_K}), we obtain, $\forall
m=0,1,\ldots $ 
\begin{equation}
\int_{-\infty }^{\infty }x\,e^{\pm mx-t\cosh x}dx=\pm m!\sum_{k=0}^{m-1}%
\frac{K_{k}\left( t\right) }{k!\left( m-k\right) }\left( \frac{t}{2}\right)
^{k-m}.  \label{Int_new_K}
\end{equation}
\end{corollary}

\section{Conclusions\label{Section: Conclusions}}

On the one hand, we have obtained new integral expressions for the $n$-th
derivatives of the Bessel functions with respect to the order in (\ref%
{Dn_Jnu}), (\ref{Dn_Ynu}), (\ref{Dn_Inu}), and (\ref{Dn_Knu}). As a
by-product, we have calculated an integral, which does not seem to be
reported in the literature, in three forms, i.e. (\ref{int_new_G}), (\ref%
{int_new_F}), and (\ref{Int_new_K}), depending on the parameter $\nu $.

On the other hand, we have derived reflection formulas for the first and
second order derivative of $J_{\nu }\left( t\right) $ in (\ref{Reflection_J}%
)\ and (\ref{Reflection_J2_closed-form}), and of $Y_{\nu }\left( t\right) $
in (\ref{Reflection_Y})\ and (\ref{Reflection_Y2_closed-form}). Also, for
arbitrary order $\nu $, we have derived the reflection formula (\ref%
{Reflection_I_general})\ for the $n$-th order derivative of $I_{\nu }\left(
t\right) $, and the reflection formula (\ref{Reflection_K}) for the $n$-th
order derivative of $K_{\nu }\left( t\right) $ in. As particular cases, for $%
I_{\nu }\left( t\right) $ and integral order $m$, we have obtained the
reflection formula (\ref{Reflection_I}) for the first order derivative, and (%
\ref{Reflection_I2_closed-form})\ for the second order derivative. Also, for
half-integral order, we have obtained formula (\ref{Reflection_I_half}).

\bibliographystyle{amsplain}

\end{document}